\documentclass[11pt,english]{amsart}
\usepackage[T1]{fontenc}
\usepackage[latin1]{inputenc}
\usepackage{amsthm}
\usepackage{amstext}
\usepackage{amssymb}

\makeatletter

\providecommand{\tabularnewline}{\\}

\numberwithin{equation}{section}
\numberwithin{figure}{section}
\theoremstyle{plain}
\newtheorem{thm}{Theorem}
  \theoremstyle{plain}
  \newtheorem{lem}[thm]{Lemma}
  \theoremstyle{remark}
  \newtheorem{rem}[thm]{Remark}
  \theoremstyle{plain}
  \newtheorem{prop}[thm]{Proposition}
  \theoremstyle{plain}
  \newtheorem{cor}[thm]{Corollary}
  \theoremstyle{remark}
  \newtheorem*{rem*}{Remark}

\makeatother

\makeatother

\usepackage{babel}

\makeatother

\usepackage{babel}

\makeatother

\usepackage{babel}

\makeatother

\usepackage{babel}

\begin{document}

\title{On the $PSL_{2}(\mathbb{F}_{19})$-invariant cubic sevenfold}

\author{Atanas Iliev, Xavier Roulleau}
\begin{abstract}
It has been proved by Adler that there exists a unique cubic hypersurface
$X^{7}$ in $\mathbb{P}^{8}$ which is invariant under the action
of the simple group $PSL_{2}(\mathbb{F}_{19})$. In the present note
we study the intermediate Jacobian of $X^{7}$ and in particular we
prove that the subjacent $85$-dimensional torus is an Abelian variety.
The symmetry group $G=PSL_{2}(\mathbb{F}_{19})$ defines uniquely
a $G$-invariant abelian $9$-fold $A(X^{7})$, which we study in
detail and describe its period lattice. 
\end{abstract}



\maketitle

\section*{Introduction}

Let $PSL_{2}(\mathbb{F}_{q})$ be the projective special linear group
of order $2$ matrices over the finite field with $q$ element $\mathbb{F}_{q}$.
There exist exactly two non trivial irreducible complex representations
$W_{\frac{q-1}{2}},$ $\bar{W}_{\frac{q-1}{2}}$ of $PSL_{2}(\mathbb{F}_{q})$
on a space of dimension $\frac{q-1}{2}$, each one complex conjugated
to the other, see \cite{Atlas}. In \cite{Adler}, Adler proved that
for any $q$ a power of a prime $p\geq7$, with $p=3\,\mbox{mod }8$,
there exists (up to a constant multiple) a unique $PSL_{2}(\mathbb{F}_{q})$-invariant
cubic form $f_{q}$ of $W_{\frac{q-1}{2}}$. We call the corresponding
unique $PSL_{2}(\mathbb{F}_{q})$-invariant cubic hypersurface\[
X^{\frac{q-5}{2}}=\{f_{q}=0\}\hookrightarrow\mathbb{\mathbb{P}}(W_{\frac{q-1}{2}})\]
the \emph{Adler cubic for $q$}.

In the smallest non-trivial case, $q=11=1\cdot8+3$, the the Adler
cubic in $\mathbb{P}^{4}=\mathbb{P}(W_{\frac{11-1}{2}})$ coincides
with the Klein cubic threefold\[
X^{3}=\{x_{1}^{2}x_{2}+x_{2}^{2}x_{3}+x_{3}^{2}x_{4}+x_{4}^{2}x_{5}+x_{5}^{2}x_{1}=0\}.\]
 This threefold has been introduced by Felix Klein as a higher dimensional
analog of the well known Klein quartic curve, the unique curve in
the projective plane with symmetry group $PSL_{2}(\mathbb{F}_{7})$,
see \cite{Klein}. In \cite{Roulleau}, the second author has proven
that there exists a unique Abelian fivefold $A(X^{3})$ with a $PSL_{2}(\mathbb{F}_{11})$-invariant
principal polarization and he has explicitly described the period
lattice of $A(X^{3})$. In this case, $A(X^{3})$ coincide with the
Griffiths intermediate Jacobian $J_{G}(X^{3})$ of the cubic threefold
$X^{3},$ with the principal polarization coming from the intersection
of real $3$-cycles on $X^{3}$.

\medskip{}

In this paper we study the next case - the Adler cubic for $q=19=2\cdot8+3$.
By using the general descriptions from Theorem 4 of \cite{Adler},
one can find an equation of the Adler cubic $X^{7}$:\[
\begin{array}{cc}
 & f_{19}=x_{1}^{2}x_{6}+x_{6}^{2}x_{2}+x_{2}^{2}x_{7}+x_{7}^{2}x_{4}+x_{4}^{2}x_{5}+x_{5}^{2}x_{8}+x_{8}^{2}x_{9}+x_{9}^{2}x_{3}+x_{3}^{2}x_{1}\\
 & -2(x_{1}x_{7}x_{8}+x_{2}x_{3}x_{5}+x_{4}x_{6}x_{9}).\end{array}\]
 \smallskip{}

In the first section, we study a similar invariant principally polarized
Abelian ninefold $A(X^{7})$ defined uniquely by the Adler cubic sevenfold,
and compute the period lattice and the first Chern class of the polarization
of $A(X^{7})$.

In the second section, we show that on the 85-dimensional complex
torus $J=(H^{5,2}\oplus H^{4,3})^{*}/H_{7}(F_{19},\mathbb{Z})$ of
the Griffiths intermediate Jacobian $J_{G}(X^{7})$ of $X^{7}$, one
can introduce a structure of a $PSL_{2}(\mathbb{F}_{19})$-invariant
polarization which provides $J$ with a structure of an Abelian variety.

In the third section, we study invariant properties of the Adler-Klein
pencil of cubic sevenfolds, an analog of the Dwork pencil of quintics
threefolds see e.g. \cite{Candelas}. Notice that from the point of
view of variation of Hodge structure, the cubic sevenfolds can be
considered as higher dimensional analogs of Calabi-Yau threefolds,
see \cite{Collino} and \cite{IlievManivel}.

 \textbf{Acknowledgements}.  We thank Bert van Geemen for its comments on a preliminary version of this paper. The second author acknowledges the hospitality of the Seoul National University where part of this work was done. He is a member of the project Geometria Algebrica PTDC/MAT/099275/2008 and was supported by FCT grant SFRH/BPD/72719/2010.

\section{The invariant Abelian $9$-fold of the Adler cubic $7$-fold.}

We begin by some notations. Let $e_{1},\dots,e_{9}$ be a basis of
a $9$-dimensional vector space $V$. Let $\tau\in GL(V)$ be the
order $19$ automorphism defined by\[
\tau:e_{j}\rightarrow\xi^{j^{2}}e_{j},\]
where $\xi=e^{2i\pi/19}$, $i^{2}=-1$, and let be $\sigma\in GL(V)$
the order $9$ automorphism defined by\[
\sigma:e_{k}\rightarrow e_{\left|6j\right|},\]
 where for any integer $t$ the notation $\left|\text{t}\right|$
means the unique integer $0\leq u\leq9$ such that $t=\pm u\mbox{ mod }19$.
We denote by $\mu$ the order $2$ automorphism given in the basis
$e_{1},\cdots,e_{9}$ by the matrix:\[
\mu=\left(\frac{i}{\sqrt{19}}\left(\frac{kj}{19}\right)(\xi^{kj}-\xi^{-kj})\right)_{1\leq k,j\leq9}\]
where $\left(\frac{kj}{19}\right)$ is the Legendre symbol. The group
generated by $\tau,\sigma$ and $\mu$ is isomorphic to $PSL_{2}(\mathbb{F}_{19})$
and define a representation of $PSL_{2}(\mathbb{F}_{19})$, see \cite{Atlas}.
Let us define\[
\begin{array}{c}
v_{k}=\tau^{k}(e_{1}+\dots+e_{9})=\\
=\xi^{k}e_{1}+\xi^{4k}e_{2}+\xi^{9k}e_{3}+\xi^{16k}e_{4}+\xi^{6k}e_{5}+\xi^{17k}e_{6}+\xi^{11k}e_{7}+\xi^{7k}e_{8}+\xi^{5k}e_{9}\end{array}\]
 (thus $v_{k}=v_{k+19}$) and\[
w_{k}'=\frac{1}{1+2\nu}(v_{k}-5v_{k+1}+10v_{k+2}-10v_{k+3}+5v_{k+4}-v_{k+5}),\]
where \[
\nu=\sum_{k=1}^{k=9}\xi^{k^{2}}=\frac{-1+i\sqrt{19}}{2}.\]

An endomorphism $h$ of a torus $A$ acts on the tangent space $T_{A,0}$
by its differential $dh$, which we call the \emph{analytic representation}
of $h$. For the ease of the notations, we will use the same letter
for $h$ and its analytic representation. 

\medskip{}

In the present section, we will prove the following Theorem:
\begin{thm}
\label{thm:The-Abelian-variety}There exist:

(1) a $9$-dimensional torus $A=V/\Lambda$, $T_{A,0}=V$ such that
the elements $\tau$, $\sigma$ of $GL(V)$ are the analytic representations
of automorphisms of $A$. The torus $A$ is isomorphic to the Abelian
variety $E^{9}$, where $E$ is the elliptic curve $\mathbb{C}/\mathbb{Z}[\nu]$,
($\nu=\frac{-1+i\sqrt{19}}{2}$).

(2) a unique principal polarization $\Theta$ on $A$ which is invariant
by the automorphisms $\sigma,\tau$. The period lattice of $A$ is
then:\[
H_{1}(A,\mathbb{Z})=\frac{\mathbb{Z}[\nu]}{1+2\nu}w_{0}'+\frac{\mathbb{Z}[\nu]}{1+2\nu}w_{1}'+\frac{\mathbb{Z}[\nu]}{1+2\nu}w_{2}'+\frac{\mathbb{Z}[\nu]}{1+2\nu}w_{3}'+\bigoplus_{k=4}^{8}\mathbb{Z}[\nu]v_{k},\]
and $c_{1}(\Theta)=\frac{2}{\sqrt{19}}\sum_{k=1}^{k=9}dx_{k}\wedge d\bar{x}_{k}$,
where $x_{1},\dots,x_{9}$ is the dual basis of the $e_{j}$'s. 
\end{thm}
In order to prove Theorem \ref{thm:The-Abelian-variety}, we first
suppose that such a torus $A=V/\Lambda$ exists and find the necessary
condition for its existence. We then check immediately that these
conditions are also sufficient, and that they give us the uniqueness
of $A$.

Let $\Lambda\subset V$ be a lattice such that $\tau$ and $\sigma$
of $GL(V)$ are analytic representations of automorphisms of the torus
$A=V/\Lambda$. Let\[
\ell_{k}=\xi^{k}x_{1}+\xi^{4k}x_{2}+\xi^{9k}x_{3}+\xi^{16k}x_{4}+\xi^{6k}x_{5}+\xi^{17k}x_{6}+\xi^{11k}x_{7}+\xi^{7k}x_{8}+\xi^{5k}x_{9}.\]
Let $q$ be the endomorphism $q=\sum_{j=0}^{j=8}\sigma^{j}$. For
$z=\sum_{j=1}^{j=9}x_{j}e_{j}$, we have $q(z)=\ell_{0}(z)v_{0}$,
thus the image of $q$ is an elliptic curve $\mathbb{E}$ contained
in $A$.

The restriction of $q\circ\tau:A\rightarrow\mathbb{E}$ to $\mathbb{E}\hookrightarrow A$
is the multiplication by $\nu=\ell_{0}(\tau v_{0})$. Thus the endomorphism
group $\mbox{End }\mathbb{E}$ of $\mathbb{E}$ contains the ring
$\mathbb{Z}[\nu]$ ; since this is a maximal order, we have $\mbox{End }\mathbb{E}=\mathbb{Z}[\nu]$.
We remark that the ring $\mathbb{Z}[\nu]$ is one of the $9$ rings
of integers of quadratic fields that are Principal Ideal Domains,
see \cite{Masley}. Since $\mathbb{Z}[\nu]$ is a PID and $H_{1}(A,\mathbb{Z})\cap\mathbb{C}v_{0}$
is a rank one $\mathbb{Z}[\nu]$-module, there exists a constant $c\in\mathbb{C}^{*}$
such that\[
H_{1}(A,\mathbb{Z})\cap\mathbb{C}v_{0}=\mathbb{Z}[\nu]cv_{0}.\]
 Up to normalization of the $e_{j}$'s, we can suppose that $c=1$.

Let $\Lambda_{0}\subset V$ be the $\mathbb{Z}$-module generated
by the $v_{k},\, k\in\mathbb{Z}/19\mathbb{Z}$. The group $\Lambda_{0}$
is stable under the action of $\tau$ and $\Lambda_{0}\subset H_{1}(A,\mathbb{Z})=\Lambda$. 
\begin{lem}
\label{A est isomorphe a E9}The $\mathbb{Z}$-module $\Lambda_{0}\subset H_{1}(A,\mathbb{Z})$
is equal to the lattice:\[
R_{0}=\sum_{k=0}^{k=8}\mathbb{Z}[\nu]v_{0}.\]
\end{lem}
\begin{proof}
We have $\nu v_{0}=\sum_{k=1}^{k=9}v_{k^{2}}$ hence $\nu v_{0}$
is an element of $\Lambda_{0}$. This implies that the vectors $\nu v_{k}=\tau{}^{k}\nu v_{0}$
are elements of $\Lambda_{0}$ for all $k$, hence: $R_{0}\subset\Lambda_{0}$.
Conversely, we have:\[
v_{9}=v_{0}+(1+\nu)v_{1}-2v_{2}+(1-\nu)v_{3}+(3+\nu)v_{4}+(-2+\nu)v_{5}-(2+\nu)v_{6}+2v_{7}+\nu v_{8}\]
 and similar formulas for $v_{10},\dots,v_{17}$. This proves that
the lattice $R_{0}$ contains the vectors $v_{k}=\tau{}^{k}v_{0}$
generating $\Lambda_{0}$, thus: $R_{0}=\Lambda_{0}$. 
\end{proof}
Now we construct a lattice that contains $H_{1}(A,\mathbb{Z})$. Let
be $k\in\mathbb{Z}/19\mathbb{Z}$. The image of $z\in V$ by the endomorphism
$q\circ\tau^{k}:V\to V$ is\[
q\circ\tau^{k}(z)=\ell_{k}(z)v_{0}.\]
 Let be $\lambda\in H_{1}(A,\mathbb{Z})$. Since $H_{1}(A,\mathbb{Z})\cap\mathbb{C}v_{0}=\mathbb{Z}[\nu]v_{0},$
the scalar $\ell_{k}(\lambda)$ is an element of $\mathbb{Z}[\nu]$.
Let \[
\Lambda_{8}=\{z\in V\,|\,\ell_{k}(z)\in\mathbb{Z}[\nu],\,0\leq k\leq8\}.\]

\begin{lem}
\label{le reseau lambda 4}The $\mathbb{Z}$-module $\Lambda_{8}\supset H_{1}(A,\mathbb{Z})$
is the lattice:\[
\sum_{k=0}^{k=7}\frac{\mathbb{Z}[\nu]}{1+2\nu}(v_{k}-v_{k+1})+\mathbb{Z}[\nu]v_{0}.\]
 Moreover $\tau$ stabilizes $\Lambda_{8}$.\end{lem}
\begin{proof}
Let $\ell_{0}^{*},\dots,\ell_{8}^{*}$ be the basis dual to $\ell_{0},\dots,\ell_{8}$.
By definition, the $\mathbb{Z}[\nu]$-module $\Lambda_{8}$ is $\oplus_{i=0}^{i=8}\mathbb{Z}[\nu]\ell_{i}^{*}$.
By expressing the $\ell_{i}^{*}$ in the basis $v_{0},\dots,v_{8}$
we obtain the lattice $\Lambda_{8}$. Using the formula: \[
\begin{array}{c}
v_{9}-v_{8}=v_{0}+(1+\nu)v_{1}-2v_{2}+(1-\nu)v_{3}+(3+\nu)v_{4}\\
+(\nu-2)v_{5}-(2+\nu)v_{6}+2v_{7}+(\nu-1)v_{8},\end{array}\]
one can check that $v_{9}-v_{8}$ is a $\mathbb{Z}[\nu]$-linear combination
of $(2\nu+1)v_{0}$ and the $v_{k}-v_{k+1}$ for $k=0,\dots,7$. Therefore
$\tau(\frac{1}{2\nu+1}(v_{8}-v_{7}))=\frac{1}{2\nu+1}(v_{9}-v_{8})$
is in $\Lambda_{8}$ and $\Lambda_{8}$ is stable by $\tau$. 
\end{proof}
We denote by $\phi:\Lambda_{8}\rightarrow\Lambda_{8}/\Lambda_{0}$
the quotient map. The ring $\mathbb{Z}[\nu]/(1+2\nu)$ is the finite
field with $19$ elements. The quotient $\Lambda_{8}/\Lambda_{0}$
is a $\mathbb{Z}[\nu]/(1+2\nu)$-vector space with basis $t_{1},\dots,t_{8}$
such that $t_{i}=\frac{1}{1+2\nu}(v_{i-1}-v_{i})+\Lambda_{0}$.

Let $R$ be a lattice such that : $\Lambda_{0}\subset R\subset\Lambda_{8}$.
The group $\phi(R)$ is a sub-vector space of $\Lambda_{8}/\Lambda_{0}$
and:\[
\phi^{-1}\phi(R)=R+\Lambda_{0}=R.\]
 The set of such lattices $R$ correspond bijectively to the set of
sub-vector spaces of $\Lambda_{8}/\Lambda_{0}$. \\
 Because the automorphism $\tau$ preserves $\Lambda_{0}$, it
induces an automorphism $\widehat{\tau}$ on the quotient $\Lambda_{8}/\Lambda_{0}$
such that $\phi\circ\tau=\widehat{\tau}\circ\phi$. As $\tau$ stabilizes
$H_{1}(A,\mathbb{Z})$, the vector space $\phi(H_{1}(A,\mathbb{Z}))$
is stable by $\widehat{\tau}$. Let\[
w_{8-k}=(-1)^{k}\sum_{j=0}^{j=k}\left(\begin{array}{c}
k\\
j\end{array}\right)(-1)^{j}t_{j+1},\, k=0\dots7.\]
Then $w_{i-1}=\widehat{\tau}w_{i}-w_{i}$ for $i>1$ and $\widehat{\tau}w_{1}=w_{1}$.
The matrix of $\widehat{\tau}$ in the basis $w_{1},.\ldots,w_{8}$
is the size $8\times8$ matrix:\[
\left(\begin{array}{ccccc}
1 & 1 & 0 & \ldots & 0\\
0 & \ddots & \ddots & \ddots & \vdots\\
\vdots & \ddots & \ddots & \ddots & 0\\
\vdots &  & \ddots & \ddots & 1\\
0 & \ldots & \ldots & 0 & 1\end{array}\right)\]
 The sub-spaces stable by $\widehat{\tau}$ are the spaces $W_{j},\,1\leq j\leq8$
generated by $w_{1},\dots,w_{j}$ and $W_{0}=\{0\}$. Let $\Lambda_{j}$
be the lattice $\phi^{-1}W_{j}$. It is easy to check that, as a lattice
in $\mathbb{C}^{9}$, the $\Lambda_{j}$ are all isomorphic to $\mathbb{Z}[\nu]^{9}$.
Since $\Lambda=H_{1}(A,\mathbb{Z})$ is stable by $\tau$, we have
proved that :
\begin{lem}
\label{lem:The-torus-}The torus $A=V/\Lambda$ exists, it is an Abelian
variety isomorphic to $E^{9}$. There exists $j\in\{0,1,\cdots,8\}$
such that $\Lambda=\Lambda_{j}$. 
\end{lem}
Let $w_{0}',\dots,w_{3}'$ be the vectors defined by:\[
w_{0}'=\frac{1}{1+2\nu}(v_{0}-5v_{1}+10v_{2}-10v_{3}+5v_{4}-v_{5})=\frac{1}{1+2\nu}(1-\tau)^{5}v_{0}\]
 and $w_{k}'=\tau^{k}w'_{0}$ (we have $\phi(w'_{0})=w_{4}$). Let
us suppose that $A=V/\Lambda$ has moreover a principal polarization
$\Theta$ that is invariant under the action of $\tau$ and $\sigma$.
Then:
\begin{lem}
\label{thmThe-attice}The lattice $H_{1}(A,\mathbb{Z})$ is equal
to $\Lambda_{4}$, and\[
\Lambda_{4}=\frac{\mathbb{Z}[\nu]}{1+2\nu}w_{0}'+\frac{\mathbb{Z}[\nu]}{1+2\nu}w_{1}'+\frac{\mathbb{Z}[\nu]}{1+2\nu}w_{2}'+\frac{\mathbb{Z}[\nu]}{1+2\nu}w_{3}'+\bigoplus_{k=4}^{8}\mathbb{Z}[\nu]v_{k}.\]
The Hermitian matrix associated to $\Theta$ is equal to $\frac{2}{\sqrt{19}}I_{9}$
in the basis $e_{1},\dots,e_{9}$ and $c_{1}(\Theta)=\frac{i}{\sqrt{19}}\sum_{k=1}^{k=9}dx_{k}\wedge d\bar{x}_{k}$.\end{lem}
\begin{proof}
Let $H$ be the matrix of the Hermitian form of the polarization $\Theta$
in the basis $e_{1},\dots,e_{9}$. Since $\tau$ preserves the polarization
$\Theta$, this implies that:\[
^{t}\tau H\bar{\tau}=H\]
where $\bar{\tau}$ is the matrix whose coefficients are conjugated
of those of $\tau$. The only Hermitian matrices that verify this
equality are the diagonal matrices. By the same reasoning with $\sigma$
instead of $\tau$, we obtain that these diagonal coefficients are
equal, and:\[
H=a\frac{2}{\sqrt{19}}I_{9},\]
 where $a$ is a positive real ($H$ is a positive definite Hermitian
form). As $H$ is a polarization, the alternating form $c_{1}(\Theta)=\Im m(H)$
takes integer values on $H_{1}(A,\mathbb{Z})$, hence $\Im m(^{t}v_{2}H\bar{v}_{1})=a$
is an integer. 

Let $c_{1}(\Theta)=\Im m(H)=i\frac{a}{\sqrt{19}}\sum dx_{k}\wedge d\bar{x}_{k}$
be the alternating form of the principal polarization $\Theta$. Let
$\lambda_{1},\dots,\lambda_{18}$ be a basis of a lattice $\Lambda'$.
By definition, the square of the Pfaffian $Pf_{\Theta}(\Lambda')$
of $\Lambda'$ is the determinant of the matrix\[
M_{\Lambda'}=\left(c_{1}(\Theta)(\lambda_{j},\lambda_{k})\right)_{1\leq j,k\leq18}.\]
 Since $\Theta$ is a principal polarization on $A$, we have $Pf_{\Theta}(H_{1}(A,\mathbb{Z}))=1$.\\
For $j\in\{0,\dots,8\}$, it is easy to find a basis of the lattice
$\Lambda_{j}$. We compute that\[
P_{j}^{2}=a^{18}19^{8-2j}\]
where $P_{j}$ its Pfaffian of $\Lambda_{j}$. As $a$ is positive,
the only possibility that $P_{j}$ equals $1$ is $j=4$ and $a=1$.
Moreover since all coefficients of the matrix $M_{\Lambda_{4}}$ are
integers, the alternating form $\frac{i}{\sqrt{19}}\sum dx_{k}\wedge d\bar{x}_{k}$
defines a principal polarization on $A$, see \cite[Chap. 4.1]{Birkenhake}.
\end{proof}
Theorem \ref{thm:The-Abelian-variety} follows from Lemmas \ref{lem:The-torus-}
and \ref{thmThe-attice}.
\begin{rem}
One can compute that the involution $\mu$ acts on the principally
polarized Abelian variety $(A,\Theta)$, and therefore that $\mbox{Aut}(A,\Theta)$
contains the group $PSL_{2}(\mathbb{F}_{19})$.
\end{rem}
{}
\begin{rem}
In\cite{Beauville2}, Beauville proves that the intermediate Jacobian
of the $S_{6}$-symmetric quartic threefolds $F$ is not a product
of Jacobians of curves (see also \cite{Beauville}). That implies,
using Clemens-Griffiths results for threefolds, the irrationality
of $F$. Using the arguments at the end of \cite{Beauville2}, one
can see that the principally polarized abelian ninefold ($A,\Theta)$
is not a product of Jacobians of curves.
\end{rem}

\section{Intermediate Jacobians of the Adler cubic sevenfold}

\subsection{Intermediate Jacobian}

For a smooth algebraic complex manifold $X$ of dimension $n=2k+1$,
let $H^{p,q}(X)$ be the Hodge $(p,q)$-cohomology spaces $H^{p,q}(X)=H^{q}(X,\Omega^{p})$,
and $h^{p,q}$ = dim $H^{p,q}(X)$ be the Hodge numbers of $X$, $0\le p+q\le2n=4k+2$.
The intersection of real $(2k+1)$-dimensional chains in $X$ defines
an integer valued quadratic form $Q_{G}$ on the integer cohomology
$H_{2k+1}(X,\mathbb{Z})$, and yields an embedding of $H_{2k+1}(X,\mathbb{Z})$
as an integer lattice in the dual space of $H^{2k+1,0}(X)+H^{2k,1}(X)+...+H^{k+1,k}(X)$.
The quotient compact complex torus\[
J_{G}(X)=(H^{2k+1,0}(X)+H^{2k,1}(X)+...+H^{k+1,k}(X))^{*}/H_{2k+1}(X,\mathbb{Z})\]
 is the \emph{Griffiths intermediate Jacobian} of the $n$-fold $X$,
see \cite[p.123]{Carlson}. By the Riemann-Hodge bilinear relations,
the quadratic form $Q_{G}$ is definite on any space $H^{p,n-p}(X)$,
and has opposite signs on $H^{p,n-p}(X)$ and $H^{p',n-p'}(X)$ if
and only if $p-p'$ odd. Moreover $\Im m(Q_{G})$ takes integral values
on $H_{2k+1}(X,\mathbb{Z})$. We call the pair $(J_{G}(X),Q_{G})$
the polarized Griffiths intermediate Jacobian. 

The quadratic form $Q_{G}$ gives the polarized torus $(J_{G}(X),Q_{G})$
a structure of a polarized Abelian variety if and only if $p$ has
the same parity for all non-zero spaces $H^{p,n-p}(X)$, (\cite[p.123]{Carlson}).

\subsection{Griffiths formulas}

The middle Hodge groups of a smooth cubic sevenfold can be computed
by the following Griffiths formulas, which we shall use below in order
to determine the representation of the action of the symmetry group
$PSL_{2}(\mathbb{F}_{19})$ on the middle cohomology of the Adler
cubic sevenfold.

Let $X=(f(x)=0)$ be a smooth hypersurface of degree $m\ge2$ in the
complex projective space $\mathbb{P}^{n+1}(x)$, $(x)=(x_{1},...,x_{n+2})$.
Let $S=\oplus_{d\ge0}S_{d}$ be the graded polynomial ring $S={\bf C}[x_{0},...,x_{n+2}]$,
with $S_{d}$ being the space of polynomials of degree $d$. Let $I=\oplus I_{d}\subset S$
be the graded ideal generated by the $n+2$ partials $\frac{\partial f}{\partial x_{j}}$,
$j=1,...,n+2$, with $I_{d}=I\cap S_{d}$.\\
Let $R=S/I=\oplus R_{d}$ be the graded Jacobian ring of $X$ (or
the Jacobian ring of the polynomial $f(x)$), with graded components
$R_{d}=S_{d}/I_{d}$. Then for $p+q=n$, the primitive cohomology
group $H_{prim}^{p,q}(X)$ is isomorphic to the graded piece $R_{m(q+1)-n-2}$,
where $R_{d}=0$ for $d<0$, (\cite{Carlson}, p. 169). For odd $n=2k+1$
all the middle cohomology of $X$ are primitive, and in this case\[
H^{p,q}(X)\cong R_{m(q+1)-n-2},\ p+q=2k+1=\mbox{ dim}X,\, m=\deg X.\]
In particular, for a smooth cubic sevenfold $X=(f(x)=0)\subset\mathbb{P}^{8}$
one has\[
H^{7-q,q}(X)=R_{3(q+1)-9},\ 0\le q\le7,\]
 which yields\[
\begin{array}{c}
H^{7,0}(X)=R_{-6}=0,\ H^{6,1}(X)=R_{-3}=0,\\
H^{5,2}(X)=R_{0}\cong\mathbb{C},\ H^{4,3}(X)=R_{3}\cong{\bf C}^{84}.\end{array}\]

\subsection{Character table of $PSL_{2}(\mathbb{F}_{19})$}

In Proposition \ref{decomposition esp. tangent} below, we shall use
the known description of the irreducible representations of the automorphism
group $PSL_{2}(\mathbb{F}_{19})\subset Aut(X)$, which we state here
in brief:

By \cite{Atlas}, the group $PSL_{2}(\mathbb{F}_{19})$ has $12$
conjugacy classes: \[
1,\{w_{1}\},\{w{}_{2}\},\{x\},\{x^{2}\},\{x^{3}\},\{x^{4}\},\{y\},\{y^{2}\},\{y^{3}\},\{y^{4}\},\{y^{5}\}\]
where $x$ has order $9$, $y$ has order $10$, $w_{1}=\left(\begin{array}{cc}
1 & 1\\
0 & 1\end{array}\right)$, $w{}_{2}=w_{1}^{2}$ (we observe that $w_{1}^{2},w_{1}^{3}\in\{w_{2}\},\, w_{2}^{2},w_{2}^{3}\in\{w_{1}\})$.
Correspondingly, the $12$ irreducible representations of $PSL_{2}(\mathbb{F}_{19})$
are \[
T_{1},W_{9},\overline{W}_{9},W_{18}^{1},W_{18}^{2},W_{18}^{3},W_{18}^{4},W_{20}^{1},W_{20}^{2},W_{20}^{3},W_{20}^{4},W_{19}\]
where the notations are uniquely defined by the character table of
$PSL_{2}(\mathbb{F}_{19})$ below for which $a_{k}=2\cos(\frac{2k\pi}{9})$,
$b_{k}=-2\cos(\frac{k\pi}{5})$ and $\nu=\frac{-1+i\sqrt{19}}{2}$.
Note that the representation $W_{9}$ is described otherwise in section
1.

\begin{figure}[h]
\begin{centering}
\begin{tabular}{|c|cccccccccccc|}
\hline 
${}$  & 1  & $w_{1}$  & $w_{2}$  & $x$  & $x^{2}$  & $x^{3}$  & $x^{4}$  & $y$  & $y^{2}$  & $y^{3}$  & $y^{4}$  & $y^{5}$ \tabularnewline
\hline 
1  & 1  & 1  & 1  & 1  & 1  & 1  & 1  & 1  & 1  & 1  & 1  & 1 \tabularnewline
$W_{9}$  & 9  & $\nu$  & $\bar{\nu}$  & 0  & 0  & 0  & 0  & 1  & -1  & 1  & -1  & 1 \tabularnewline
${\overline{W}_{9}}$  & 9  & $\bar{\nu}$  & $\nu$  & 0  & 0  & 0  & 0  & 1  & -1  & 1  & -1  & 1 \tabularnewline
$W_{18}^{1}$  & 18  & -1  & -1  & 0  & 0  & 0  & 0  & $b_{1}$  & $b_{2}$  & $b_{3}$  & $b_{4}$  & 2 \tabularnewline
$W_{18}^{2}$  & 18  & -1  & -1  & 0  & 0  & 0  & 0  & $b_{2}$  & $b_{4}$  & $b_{6}$  & $b_{8}$  & -2 \tabularnewline
$W_{18}^{3}$  & 18  & -1  & -1  & 0  & 0  & 0  & 0  & $b_{3}$  & $b_{6}$  & $b_{9}$  & $b_{12}$  & 2 \tabularnewline
$W_{18}^{4}$  & 18  & -1  & -1  & 0  & 0  & 0  & 0  & $b_{4}$  & $b_{8}$  & $b_{12}$  & $b_{16}$  & -2 \tabularnewline
$W_{20}^{1}$  & 20  & 1  & 1  & $a_{1}$  & $a_{2}$  & $a_{3}$  & $a_{4}$  & 0  & 0  & 0  & 0  & 0 \tabularnewline
$W_{20}^{2}$  & 20  & 1  & 1  & $a_{2}$  & $a_{4}$  & $a_{6}$  & $a_{8}$  & 0  & 0  & 0  & 0  & 0 \tabularnewline
$W_{20}^{3}$  & 20  & 1  & 1  & $a_{3}$  & $a_{6}$  & $a_{9}$  & $a_{12}$  & 0  & 0  & 0  & 0  & 0 \tabularnewline
$W_{20}^{4}$  & 20  & 1  & 1  & $a_{4}$  & $a_{8}$  & $a_{12}$  & $a_{16}$  & 0  & 0  & 0  & 0  & 0 \tabularnewline
$W_{19}$  & 19  & 0  & 0  & 1  & 1  & 1  & 1  & -1  & -1  & -1  & -1  & -1 \tabularnewline
\hline
\end{tabular}
\par\end{centering}

\caption{Character table of $PSL_{2}(\mathbb{F}_{19})$}

\label{fig:wam} 
\end{figure}

\subsection{Periods of the Adler cubic}

Let $X\subset\mathbb{P}^{8}=\mathbb{P}(W_{9})$ be the Adler cubic
sevenfold\[
\begin{array}{cc}
\{f_{19}= & x_{1}^{2}x_{6}+x_{6}^{2}x_{2}+x_{2}^{2}x_{7}+x_{7}^{2}x_{4}+x_{4}^{2}x_{5}+x_{5}^{2}x_{8}+x_{8}^{2}x_{9}+x_{9}^{2}x_{3}+x_{3}^{2}x_{1}\\
 & -2(x_{1}x_{7}x_{8}+x_{2}x_{3}x_{5}+x_{4}x_{6}x_{9})=0\},\end{array}\]
 and let $H^{p,q}=H^{p,q}(X)$ be the Hodge cohomology spaces of $X$. 
\begin{prop}
\label{decomposition esp. tangent} The representation of the group
$PSL_{2}(\mathbb{F}_{19})$ on the 84-dimensional space ${H^{4,3}(X)}^{*}$
of the Adler cubic $X$ is:\[
H^{4,3}(X)^{*}=\overline{W}_{9}\oplus W_{18}^{1}\oplus W_{18}^{3}\oplus W_{19}\oplus W_{20}^{3}.\]
The group $PSL_{2}(\mathbb{F}_{19})$ acts trivially on the one-dimensional
space $H^{5,2}(X)^{*}$.\end{prop}
\begin{rem}
\label{rem:dual}Let a group $G$ acts on a vector space $V$ on the
left. Then the group $G$ acts on the right on the dual space $V^{*}$:
for $\ell\in V^{*}$ and $g\in G$, $\ell\cdot g=\ell\circ g$. Since
the traces of the action of $g$ on $V$ and on $V^{*}$ are equal,
the two representations $V$ and $V^{*}$ are isomorphic. The representation
$V^{*}$ should not be confused with the dual representation of $G$
defined such that the pairing between $V$ and $V^{*}$ is $G$-invariant
(see \cite{Fulton}).\end{rem}
\begin{proof}
In order to decompose the representation of $PSL_{2}(\mathbb{F}_{19})\subset Aut(X)$
on the dual cohomology space ${H^{4,3}}^{*}$, we shall use the identification
as representation space\[
H^{4,3}(X)\cong R_{3}=S_{3}/I_{3}=\mbox{Sym}^{3}W_{9}/I_{3}\]
between $H^{4,3}(X)$ and the graded component of degree $3$ in the
quotient polynomial ring $S={\bf C}[x_{1},...,x_{9}]$ by the Jacobian
ideal $I$ spanned by the 9 partials of $f_{19}$, see above.

The space $\mathbb{P}^{8}$ containing the Adler cubic $X$ is the
projectivization of the representation space $W_{9}=(S_{1})^{*}$.
We have $S_{3}=\mbox{Sym}^{3}(W_{9}^{*})\cong\mbox{Sym}^{3}(W_{9})$.
Let us decompose $\mbox{Sym}^{3}W_{9}$ into irreducible representations
of $PSL_{2}(\mathbb{F}_{19})$. The character of the third symmetric
power of a representation $V$ is:\[
\chi_{\text{Sym}^{3}V}(g)=\frac{1}{6}(\chi_{V}(g)^{3}+3\chi_{V}(g^{2})\chi_{V}(g)+2\chi_{V}(g^{3})),\]
see \cite{Pulay}. Therefore the traces of the action of the elements
$1,w_{1},w_{2}$ etc on $\mbox{Sym}^{3}W_{9}$ are\[
v=\,^{t}(165,3-\nu,3-\bar{\nu},0,0,3,0,0,0,0,0,5).\]
 Using the character table of $PSL_{2}(\mathbb{F}_{19})$, we obtain:\[
\mbox{Sym}^{3}W_{9}=T_{1}\oplus\overline{W}_{9}\oplus W_{18}^{1}\oplus W_{18}^{3}\oplus W_{20}^{1}\oplus W_{20}^{2}\oplus(W_{20}^{3})^{\oplus2}\oplus W_{20}^{4}\oplus W_{19}.\]

The graded component ${I}_{2}$ of the Jacobian ideal $I$ of $X$
is generated by the $9$ derivatives $\frac{df_{19}}{dx_{k}},\, k=1,\dots,9$.
The space $I_{2}$ is a representation of $PSL_{2}(\mathbb{F}_{19})$.
The action of $w_{1}$ on $x_{j}$ is the multiplication by $\xi^{j^{2}}$,
where $\xi=e^{2i\pi/19}$, see section $1$. The action of $w_{1}$
on $I_{2}$ is then easy to compute. By example, since \[
\frac{df_{19}}{dx_{1}}=2x_{1}x_{6}+x_{3}^{2}-2x_{7}x_{8},\]
we get $w_{1}\cdot\frac{df_{19}}{dx_{1}}=\xi^{-1}\frac{df_{19}}{dx_{1}}$.
By looking at the character table, we obtain that the representation
${I}_{2}$ is $\overline{W}_{9}$. Therefore ${I}_{3}=W_{9}\otimes\overline{W}_{9}$.
Using the fact that for two representations $V_{1},V_{2}$, their
characters satisfy the relation $\chi_{V_{1}\otimes V_{2}}=\chi_{V_{1}}\chi_{V_{2}}$,
we obtain: \[
I_{3}=W_{9}\otimes\overline{W}_{9}=T_{1}\oplus W_{20}^{1}\oplus W_{20}^{2}\oplus W_{20}^{3}\oplus W_{20}^{4}.\]
Since\[
{H^{4,3}(X)}^{*}\cong(S_{3}/I_{3})^{*}\cong\mbox{Sym}^{3}W_{9}/I_{3}\]
then\[
H^{4,3}(X)^{*}=\overline{W}_{9}\oplus W_{18}^{1}\oplus W_{18}^{3}\oplus W_{19}\oplus W_{20}^{3}.\]

The action of the group $PSL_{2}(\mathbb{F}_{19})\subset Aut(X)$
on $H^{5,2}(X)^{*}\cong R_{0}^{*}\cong{\bf C}$ is trivial, see the
character table. 
\end{proof}
Let $V$ be a representation of a finite group $G$ and let $W$ be
an irreducible representation of $G$ of character $\chi_{W}$. We
know that there exist a uniquely determined integer $a\geq0$ and
a representation $V'$ of $G$ such that $W$ is not a sub-representation
of $V'$ and $V$ is (isomorphic to) $W^{\oplus a}\oplus V'$. By
\cite{Fulton} formula (2.31) p. 23, the linear endomorphism\[
\psi_{W}=\frac{\dim W}{|G|}\sum_{g\in G}\overline{\chi_{W}(g)}g\,:\, V\to V\]
 is the projection of $V$ onto the factor consisting of the sum of
all copies of $W$ appearing in $V$ i.e. is the projection onto $W^{\oplus a}$.
For the trivial representation $T$,\ $\chi_{T}=1$ and $\psi_{T}$
is the projection of $V$ onto the invariant space $V^{G}$.

Recall that an endomorphism $h$ of a torus $Y$ acts on the tangent
space $TY$ of $Y$ by a linear endomorphism $dh$ called the analytic
representation of $h$. The kernel of the map $End(Y)\to End(TY)$,
$h\to dh$ is the group of translations (see \cite{Birkenhake}) ;
in the following we will work with endomorphisms up to translation. 

Suppose that for some irreducible representations $W_{1},\dots,W_{k}$
the sum\[
\overline{\chi_{W_{1}}(g)}+\dots+\overline{\chi_{W_{k}}(g)}\]
 is an integer for every $g$. Then the endomorphism \[
h=\sum_{g\in G}(\overline{\chi_{W_{1}}(g)}+\dots+\overline{\chi_{W_{k}}(g)})g\in End(J_{G}X)\]
 is well defined and its analytic representation is\[
dh=\sum_{g\in G}(\overline{\chi_{W_{1}}(g)}+\dots+\overline{\chi_{W_{k}}(g)})dg\in End(TJ_{G}X).\]
 The tangent space of the image of $h$ (translated to $0$) is the
image of $dh$.
\begin{cor}
\label{JF de Adler} The torus $J_{G}X$ has the structure of an Abelian
variety and is isogenous to:\[
E\times A_{9}\times A_{36}\times A_{19}\times A_{20}\]
 where $A_{k}$ is a $k$-dimensional Abelian subvariety of $J_{G}X$
and $E\subset J_{G}X$ an elliptic curve. The group $PSL_{2}(\mathbb{F}_{19})$
acts nontrivially on each factors $A_{9},A_{36},A_{19},A_{20}$. \end{cor}
\begin{proof}
The character of the trivial representation is $\chi_{0}=1$. Let
$\chi_{1},\chi_{2},\chi_{3},\chi_{4}$ be respectively the characters
of the representations $W_{9}\oplus\overline{W}_{9}$, $W_{18}^{1}\oplus W_{18}^{2}\oplus W_{18}^{3}\oplus W_{18}^{4}$,
$W_{19}$ and $W_{20}^{1}\oplus W_{20}^{2}\oplus W_{20}^{3}\oplus W_{20}^{4}$.
By the character table, the numbers $\chi_{k}(g),\, g\in PSL_{2}(\mathbb{F}_{19})$
are integers, therefore we can define\[
q_{k}=\sum_{g\in PSL_{2}(\mathbb{F}_{19})}\chi_{k}(g)g\]
for $k=0,\dots,4$. The analytic representation of $q_{k}$ is a multiple
of the projection onto the subspace $(H^{5,2})^{*},W_{9},...$ of
$TJ_{G}X$. The images of $q_{k},\, k=0,\cdots,4$ are therefore respectively
$1,9,36,19,20$-dimensional sub-tori of $J_{G}X$, stable by the action
of the group ring $\mathbb{Z}[PSL_{2}(\mathbb{F}_{19})]$ and denoted
respectively by $E,A_{9},\dots,A_{20}$. 

The image of the endomorphism $q_{1}+\dots+q_{4}$ (resp. $q_{0}$)
is a sub-torus whose tangent space is $ $$H^{4,3}(X)^{*}$ (resp.
$H^{5,2}(X)^{*}$), therefore $H_{7}(X,\mathbb{Z})\cap H^{4,3}(X)^{*}$
is a lattice in $H^{4,3}(X)^{*}$ (resp. $H_{7}(X,\mathbb{Z})\cap H^{5,2}(X)^{*}$
is a lattice in $H^{5,2}(X)^{*}$). 

Let $Q_{G}$ be the Hodge-Riemann form on the tangent space $(H^{5,2}\oplus H^{4,3})^{*}$.
It is positive definite on $H^{5,2}(X)^{*}$, negative definite on
$H^{4,3}(X)^{*}$, the space $H^{4,3}(X)^{*}$ is orthogonal to $H^{5,2}(X)^{*}$
with respect to $Q_{G}$ and $Q_{G}$ takes integral values on $H_{7}(X,\mathbb{Z})\subset(H^{5,2}\oplus H^{4,3})^{*}$
(see \cite{Carlson}, 114).

Let us define the quadratic form $Q'$ on $(H^{5,2}\oplus H^{4,3})^{*}$
by $Q'=-Q_{G}$ on $H^{4,3}(X)^{*}$ and $Q'=Q_{G}$ on $H^{5,2}(X)^{*}$.
This $Q'$ is a definite quadratic form that takes integral values
on the lattice $\Lambda\subset H^{5,2}(X)^{*}$ generated by $H_{7}(X,\mathbb{Z})\cap H^{4,3}(X)^{*}$
and $H_{7}(X,\mathbb{Z})\cap H^{5,2}(X)^{*}$. We thus see that $J'=(H^{5,2}\oplus H^{4,3})^{*}/\Lambda$
is an Abelian variety, and since $J_{G}X$ is isogenous to $J'$,
$J_{G}X$ is also an Abelian variety. 
\end{proof}

\section{On the Adler-Klein pencil of cubics.}

Here we study the Adler-Klein pencil of cubics $X_{\lambda}=\{f_{\lambda}=0\}$,\[
\begin{array}{cc}
f_{\lambda}= & x_{1}^{2}x_{6}+x_{6}^{2}x_{2}+x_{2}^{2}x_{7}+x_{7}^{2}x_{4}+x_{4}^{2}x_{5}+x_{5}^{2}x_{8}+x_{8}^{2}x_{9}+x_{9}^{2}x_{3}+x_{3}^{2}x_{1}\\
 & +\lambda(x_{1}x_{7}x_{8}+x_{2}x_{3}x_{5}+x_{4}x_{6}x_{9}),\end{array}\]
where $X_{-2}$ is the Adler cubic, $X_{0}$ is the Klein cubic. Since
$X_{-2}$ and $X_{0}$ are smooth, the general member of the pencil
is smooth.

The automorphism group of $X_{\lambda}$ contains the group \[
H=\mathbb{Z}/9\mathbb{Z}\rtimes\mathbb{Z}/19\mathbb{Z}\]
 whose law is defined multiplicatively by:\[
(a,b)(c,d)=(a+c,4^{c}\cdot b+d).\]

\begin{rem*}
The group $H$ is a subgroup of $PSL_{2}(\mathbb{F}_{19})$ : it is
the stabilizer of a point in the projective line $\mathbb{P}^{1}(\mathbb{F}_{19})$
for the action of the simple group $PSL_{2}(\mathbb{F}_{19})$. Since
$\mathbb{P}^{1}(\mathbb{F}_{19})$ has $20$ points, there are therefore
$20$ such subgroups in $PSL_{2}(\mathbb{F}_{19})$. 
\end{rem*}
Let us study the representation of $H$ on the tangent space of the
intermediate Jacobian of $X_{\lambda}$. Let be $a=(1,0)$ and $b=(0,1)\in H$.
Let us denote by $C_{g}$ the conjugacy class of an element $g\in H$.
The conjugacy classes of the group $H$ are the $11$ classes: $C_{b}$,
$C_{b^{2}}$ and $C_{a^{k}}$, $k=0,\cdots,8$. Let $\mu$ be a $9^{th}$-primitive
root of unity. The group $\mathbb{Z}/9\mathbb{Z}$ is the quotient
of $H$ by $\mathbb{Z}/19\mathbb{Z}$, thus the irreducible one-dimensional
representation \[
\chi_{k}:\left|\begin{array}{ccc}
\mathbb{Z}/9\mathbb{Z} & \rightarrow & \mathbb{C}^{*}\\
a & \rightarrow & \mu^{ka}\end{array}\right.,\, k\in\{0,\dots,8\}\]
 induces an irreducible representation $V_{k}$ of $H$. Let $V_{9},\,\overline{V}_{9}$
be the restrictions of the representations $W_{9},\,\overline{W}_{9}$
to $H\subset PSL_{2}(\mathbb{F}_{19})$. As $171=1^{2}+\dots+1^{2}+9^{2}+9^{2}$
is equal to the order of $H$, the representations $V_{0},\dots,V_{8},V_{9},\overline{V}_{9}$
are the $11$ irreducible non-isomorphic representations of $H$ (see
\cite{Fulton}). 
\begin{prop}
\label{Representation par restriction}The restrictions of the representations
$W_{18}^{3},\, W_{19},\, W_{20}$ of $PSL_{2}(\mathbb{F}_{19})$ to
$H\subset PSL_{2}(\mathbb{F}_{19})$ are decomposed as follows:\[
\begin{array}{c}
W_{18}^{3}=V_{9}+\overline{V}_{9}\\
W_{19}=V_{0}+V_{9}+\overline{V}_{9}\\
W_{20}=V_{3}+V_{6}+V_{9}+\overline{V}_{9}.\end{array}\]

\end{prop}
\begin{figure}[h]

\begin{centering}
\begin{tabular}{|c|ccccccccccc|}
\hline 
${}$  & 1  & $a$  & $a^{2}$  & $a^{3}$  & $a^{4}$  & $a^{5}$  & $a^{6}$  & $a^{7}$  & $a^{8}$  & $b$  & $b^{2}$ \tabularnewline
\hline 
$V_{0}$  & 1  & 1  & 1  & 1  & 1  & 1  & 1  & 1  & 1  & 1  & 1 \tabularnewline
$V_{1}$  & 1  & $\mu$  & $\mu^{2}$  & $\cdots$  & $\cdots$  & $\cdots$  & $\cdots$  & $\cdots$  & $\mu^{8}$  & 1  & 1 \tabularnewline
$V_{2}$  & 1  & $\mu^{2}$  & $\mu^{4}$  & ...  &  &  &  &  & $\vdots$  & 1  & 1 \tabularnewline
$V_{3}$  & 1  & $\vdots$  &  & $\ddots$  &  &  &  &  &  & 1  & 1 \tabularnewline
$V_{4}$  & 1  & $\vdots$  &  &  &  &  &  &  &  & 1  & 1 \tabularnewline
$V_{5}$  & 1  & $\vdots$  &  &  &  &  &  &  &  & 1  & 1 \tabularnewline
$V_{6}$  & 1  & $\vdots$  &  &  &  &  &  &  &  & 1  & 1 \tabularnewline
$V_{7}$  & 1  & $\vdots$  & $\vdots$  &  &  &  &  & $\ddots$  & $\vdots$  & 1  & 1 \tabularnewline
$V_{8}$  & 1  & $\mu^{8}$  & $\mu^{16}$  & $\cdots$  & $\cdots$  & $\cdots$  & $\cdots$  & $\cdots$  & $\mu^{64}$  & 1  & 1 \tabularnewline
$V_{9}$  & 9  & 0  & 0  & 0  & 0  & 0  & 0  & 0  & 0  & $\nu$  & $\bar{\nu}$ \tabularnewline
$\overline{V}_{9}$  & 9  & 0  & 0  & 0  & 0  & 0  & 0  & 0  & 0  & $\bar{\nu}$  & $\nu$ \tabularnewline
\hline
\end{tabular}
\par\end{centering}

\caption{Character table of $H=\mathbb{Z}/9\mathbb{Z}\rtimes\mathbb{Z}/19\mathbb{Z}$.
$\nu=\frac{1}{2}(-1+i\sqrt{19})$, $\mu^{9}=1$, $a=(1,0),\, b=(0,1)$. }

\label{fig:wam} 
\end{figure}

\begin{proof}
There exist positive integers $a_{0},\dots,a_{9},b_{0}$ such that
the representation of $H$ on $W_{18}^{3}$ equals (in fact, is isomorphic)
to $V_{0}^{\oplus a_{0}}\oplus\cdots\oplus V_{9}^{\oplus a_{9}}\oplus\overline{V}_{9}^{\oplus b_{9}}$.
For an element $g\in H$ acting on a vector space $V$, we denote
by $Tr(g|V)$ the trace of $g$ acting on $V$. We have\[
Tr(g|W_{18}^{3})=a_{0}Tr(g|V_{0})+\cdots+a_{9}Tr(g|V_{9})+b_{9}Tr(g|\overline{V}_{9}).\]
The values of $Tr(g|V_{1}),\cdots,Tr(g|V_{9}),\, Tr(g|\overline{V}_{9})$
are given in the above character table, and are independent of the
conjugacy class of $g$. Since we know the values of $Tr(g|W_{18}^{3})$,
$g\in H$ from the character table of $PSL_{2}(\mathbb{F}_{19})$,
by elementary linear algebra we obtain that $a_{0}=\cdots=a_{8}=0$
and $a_{9}=b_{9}=1$, therefore $W_{18}^{3}=V_{9}\oplus\overline{V}_{9}.$
In the same way, we find $W_{19}=V_{0}+V_{9}+\overline{V}_{9}$ and
$W_{20}=V_{3}+V_{6}+V_{9}+\overline{V}_{9}.$
\end{proof}
Let $X_{\lambda}$ be a smooth cubic in the Klein-Adler pencil and
let $J_{G}X_{\lambda}$ be the Griffiths intermediate Jacobian of
$X_{\lambda}$. 
\begin{cor}
\label{Repre de G sur JF(lambda)}The representation of $H$ on the
tangent space to the Griffiths intermediate Jacobian of $X_{\lambda}$
is:\[
TJ_{G}X_{\lambda}=V_{0}+H^{4,3}(X)^{*}=2V_{0}+V_{3}+V_{6}+4V_{9}+5\overline{V}_{9}.\]
 There exist subtori $A_{2}$, $B_{2}$, $B_{81}$ of $J_{G}X_{\lambda}$
of dimension respectively $2$, $2$ , $81$, such that $ $$B_{2}$
and $B_{81}$ are Abelian varieties and such that there is an isogeny
of complex tori\[
J_{G}X_{\lambda}\rightarrow A_{2}\times B_{2}\times B_{81}.\]
\end{cor}
\begin{proof}
For the Adler cubic $X_{Ad}$, the representation of $PSL_{2}(\mathbb{F}_{19})$
on $H^{4,3}(X_{Ad})$ is \[
H^{4,3}(X_{Ad})=\overline{W}_{9}\oplus W_{18}^{1}\oplus W_{18}^{3}\oplus W_{19}\oplus W_{20}^{3}.\]
From Proposition \ref{Representation par restriction}, we know the
representation of $H$ on each of the factors of $H^{4,3}$. Since
$H$ acts also on each $X_{\lambda}$, the representation of $H$
on $H^{4,3}(X_{\lambda})$ is the same for all cubics $X_{\lambda}$
in the pencil, and we have: \[
R_{3}=\mbox{S}^{3}V_{9}/{I}_{3}=V_{0}+V_{3}+V_{6}+4V_{9}+5\overline{V}_{9}.\]
 Thus $TJ_{G}X_{\lambda}=(V_{0})^{\oplus2}\oplus V_{83}$. The image
of the endomorphism $\sum_{g\in G}g$ is therefore a $2$-dimensional
torus $A_{2}$. By considering the quotient map $J_{G}X_{\lambda}\rightarrow J_{G}X_{\lambda}/A_{2}$,
we see that $V_{83}\cap H_{7}(X_{\lambda},\mathbb{Z})$ is a lattice.
There is moreover on it a positive definite integral valued form (the
restriction of the Hodge-Riemann form), therefore $V_{83}$ is the
tangent space of a $83$-dimensional Abelian subvariety $A_{83}$
of $J_{G}X_{\lambda}$. By using the same arguments as in section
2, we see that $A_{83}$ is isogenous to a product $B_{2}\times B_{81}$
of two abelian sub-varieties.
\end{proof}

\bigskip{}

\bigskip{}

{\scriptsize 
}{\scriptsize \par}

{\scriptsize Atanas Iliev}{\scriptsize \par}

{\scriptsize Department of Mathematics}{\scriptsize \par}

{\scriptsize Seoul National University}{\scriptsize \par}

{\scriptsize 151-747 Seoul, Korea}{\scriptsize \par}

{\scriptsize e-mail: }\texttt{\scriptsize ailiev@snu.ac.kr}{\scriptsize \par}

{\scriptsize \bigskip{}
 }{\scriptsize \par}

{\scriptsize 
}{\scriptsize \par}

{\scriptsize Xavier Roulleau}{\scriptsize \par}

{\scriptsize Laboratoire de Mathématiques et Applications,}{\scriptsize \par}

{\scriptsize Téléport 2 - BP 30179 -}{\scriptsize \par}

{\scriptsize Boulevard Pierre et Marie Curie }{\scriptsize \par}

{\scriptsize 86962 Futuroscope Chasseneuil}{\scriptsize \par}

{\scriptsize France}{\scriptsize \par}

{\scriptsize e-mail: }\texttt{\scriptsize xavier.roulleau@math.}{\scriptsize univ-poitiers.fr}{\scriptsize \par}


\end{document}